\documentclass[12pt,reqno]{amsart}
\usepackage{amssymb}
\usepackage{amsfonts}
\usepackage{amsmath}
\usepackage{hyperref}
\usepackage{xcolor}

\newtheorem{corollary}{Corollary}
\newtheorem{theorem}{Theorem}

\begin{document}
\title[Harmonic error functions]{On the inclusion properties for harmonic
error functions}

\begin{abstract}
For the error functions of the form
\begin{equation*}
E_{r}\mathfrak{f}(z)=\frac{\sqrt{\pi z}}{2}er\ \mathfrak{f}(\sqrt{z})=z+\Sigma_{n=2}^{\infty} \frac{(-1)^{n-1}}{(2n-1)(n-1)!}z^{n},
\end{equation*}%
let $\mathcal{E}S_{\mathcal{H}}(k,\lambda ,\gamma )\,$\ represent the class
of harmonic error functions $\mathcal{ERF}=\mathcal{ERH}+\overline{\mathcal{%
ERG}}$ in the open unit disk $\mathbb{U}=\left\{ z\in \mathbb{C}:\ \
\left\vert z\right\vert <1\right\} $. The paper attempts to
present some basic properties for functions
in this class.
\end{abstract}

\author[\c{S}. Alt\i nkaya, S. Yal\c{c}\i n]{\c{S}ahsene Alt\i nkaya$^{1,\ast }$, Sibel Yal\c{c}\i n$^{2}$}
\address{$^{1}$Department of Mathematics and Statistics, Faculty of Science, Turku University, 20014, Turku, Finland \\
\newline $^{2}$Department of Mathematics, Faculty of Arts and
Science, Bursa Uludag University, 16059, G\"{o}r\"{u}kle, Bursa, Türkiye}
\email{sahsenealtinkaya@gmail.com}
\email{syalcin@uludag.edu.tr}
\maketitle

\section{Introduction and background}

\thispagestyle{empty}Indicate by $\mathcal{A}$ the family of functions
analytic in $\mathbb{U}$, of the form: 
\begin{equation}
f(z)=z+\Sigma_{n=2}^{\infty}a_{n}z^{n},  \label{eq1}
\end{equation}%
which fulfill the normalization $f(0)=f^{\prime }(0)-1=0$ and also indicate by $\mathcal{S}$ the subfamily of $\mathcal{A}$ including univalent functions in $\mathbb{U}$. Further, for the function $g$ with the Taylor series $g(z)=z+b_{2}z^{2}+\cdots ,$ the
convolution $f\ast g$ is expressed by
\begin{equation*}
\left( f\ast g\right) (z)=z+\Sigma_{n=2}^{\infty}
a_{n}b_{n}z^{n}.
\end{equation*}%
The real-valued function $v$ is named harmonic in a domain $B\subset \mathbb{C}$ if it has continuous second order partial derivatives in $B$, which
fulfills the Laplace equation 
\begin{equation*}
\Delta v:=\frac{\partial ^{2}v}{\partial x^{2}}+\frac{\partial ^{2}v}{%
\partial y^{2}}=0.
\end{equation*}

A harmonic mapping $f$ of the simply connected domain $B$ is a
complex-valued function of the form $f=h+\overline{g}$ $\left( h\text{,}g%
\text{ analytic and }h(0)=h^{\prime }(0)-1=0,\ g(0)=0\right) .$ We call $h$
and $g$ analytic and co-analytic part of $f$, respectively (see \cite{CS-S}%
). $J_{f(z)}=\left\vert f_{z}(z)\right\vert
^{2}-\left\vert f_{\overline{z}}(z)\right\vert ^{2}=\left\vert h^{\prime
}(z)\right\vert ^{2}-\left\vert g^{\prime }(z)\right\vert ^{2}$ is defined as the Jacobian of $f$. Also, $f$ is locally univalent iff its Jacobian is
never zero, and is sense-preserving provided that the Jacobian is positive \cite{L}.

Indicate by $\mathcal{H}$ the family of all harmonic functions of the form $%
f=h+\overline{g}$, where 
\begin{equation}
h(z)=z+ \Sigma_{n=2}^{\infty}a_{n}z^{n},~\
g(z)= \Sigma_{n=1}^{\infty}b_{n}z^{n}  \label{C}
\end{equation}%
and indicate by $\mathcal{SH}$ the family of complex-valued harmonic, univalent
mappings that are normalized with $f(0)=f_{z}(0)-1=0$ in $\mathbb{U}$. Since $f=h+\overline{g}$ where $h$ and $g$ are analytic, $f$ has the series
representation 
\begin{equation}
f(z)=z+\Sigma_{n=2}^{\infty}a_{n}z^{n}+\Sigma_{n=1}^{\infty}
\overline{b_{n}z^{n}}\ \ \ \left( \left\vert b_{1}\right\vert <1,\ z\in 
\mathbb{U}\right) ,  \label{3}
\end{equation}%
which are univalent, sense-preserving in ${\mathbb{U}}.$ The subfamily $\mathcal{SH}^{0}$ of $\mathcal{SH}$ consists of all functions in $\mathcal{SH}$ which satisfying $f_{\bar{z}}(0)=b_{1}=0.$ Note that%
\begin{equation*}
\mathcal{S}\subset \mathcal{SH}^{0}\subset \mathcal{SH}.
\end{equation*}%
In 1984, Clunie and Sheil-Small \cite{CS-S} examined the family $\mathcal{SH}$ and its geometric subfamilies.
Since then, there have been a large number of papers on $\mathcal{SH}$ as well
as its subfamilies \cite{alt}, \cite{ca}, \cite{D}, \cite%
{DD}, \cite{J}, \cite{mos}, \cite{s}, \cite{yasar2}). When the co-analytic part of 
$f$ is identically zero, $\mathcal{SH}$ reduces to $\mathcal{S}$.

Al-Shaqsi and Darus \cite{al} established the derivative operator
\begin{equation*}
C_{\lambda }^{k}f(z)=C_{\lambda }^{k}h(z)+\overline{C_{\lambda }^{k}g(z)}~\
\ \ \ \ \ \ \left( k,\lambda \in \mathbb{N}
_{0}=\mathbb{N}\cup \left\{ 0\right\} \right) ,
\end{equation*}%
where $C_{\lambda }^{k}h(z)=z+\Sigma_{n=2}^{\infty}\binom{n+\lambda -1%
}{\lambda }n^{k}a_{n}z^{n}$ and $C_{\lambda
}^{k}g(z)=\Sigma_{n=1}^{\infty}\binom{n+\lambda -1}{\lambda }%
n^{k}b_{n}z^{n}.$

Recently, Ramachandran et al. \cite{ram} studied the normalized analytic
error function 
\begin{equation*}
E_{r}\mathfrak{f}(z)=\frac{\sqrt{\pi z}}{2}er\ \mathfrak{f}(\sqrt{z})=z+\Sigma_{n=2}^{\infty}\tfrac{(-1)^{n-1}}{(2n-1)(n-1)!}z^{n}
\end{equation*}%
and defined the family
\begin{equation*}
\mathcal{E}=\mathcal{A}\ast E_{r}\mathfrak{f}=\left\{ \mathcal{ERF}:\mathcal{%
ERF}(z)=\left( f\ast E_{r}\mathfrak{f}\right) (z)=z+\Sigma_{n=2}^{\infty}\tfrac{(-1)^{n-1}}{(2n-1)(n-1)!}a_{n}z^{n},\ f\in 
\mathcal{A}\right\} .
\end{equation*}%
Let $C_{\lambda }^{k}\mathcal{ERF}=C_{\lambda }^{k}\mathcal{ERH}+C_{\lambda
}^{k}\overline{\mathcal{ERG}}$ with $\mathcal{H}$ and $\mathcal{G}$ be
analytic in $\mathbb{U}$, where%
\begin{equation*}
C_{\lambda }^{k}\mathcal{ERH}(z)=z+\Sigma_{n=2}^{\infty}\tbinom{n+\lambda -1}{\lambda }\tfrac{(-1)^{n-1}n^{k}}{(2n-1)(n-1)!}a_{n}z^{n}
\end{equation*}%
and%
\begin{equation*}
C_{\lambda }^{k}\mathcal{ERG}(z)=\Sigma_{n=1}^{\infty}\tbinom{n+\lambda
-1}{\lambda }\tfrac{(-1)^{n-1}n^{k}}{(2n-1)(n-1)!}b_{n}z^{n}.
\end{equation*}

Let $\mathcal{E}S_{\mathcal{H}}(k,\lambda ,\gamma )\,$represent the family
of harmonic error functions $f$ of the form $(\ref{C})$ such that 
\begin{equation}
\Re \left\{ \frac{z\left( C_{\lambda }^{k}\mathcal{ERF}(z)\right) _{z}-
\overline{z}\left( C_{\lambda }^{k}\mathcal{ERF}(z)\right) _{\overline{z}}}{C_{\lambda }^{k}\mathcal{ERF}(z)}\right\} \geq \gamma ~\ \left( k,\lambda
\in \mathbb{N},\ 0\leq \gamma <1\right) .  \label{5}
\end{equation}%
Let $\mathcal{E}S_{\overline{\mathcal{H}}}(k,\lambda ,\gamma )\,$represent
the subfamily of $\mathcal{E}S_{\mathcal{H}}(k,\lambda ,\gamma )$ consists of
harmonic functions $f=h+\overline{g}$ such that $h$, $g$ are of the form 
\begin{equation}
h(z)=z+\Sigma_{n=2}^{\infty}(-1)^{n}\left\vert a_{n}\right\vert
z^{n},~\ g(z)=\Sigma_{n=1}^{\infty}(-1)^{n-1}\left\vert
b_{n}\right\vert z^{n}.  \label{x}
\end{equation}

\section{Main results}

Firstly, we attempt to find the sufficient condition for harmonic error
functions in $\mathcal{E}S_{\mathcal{H}}(k,\lambda ,\gamma ).$

\begin{theorem}
\label{t1}If a function $f\in \mathcal{H}$ of the form $(\ref{C})$ fulfills
\begin{equation}
\Sigma_{n=1}^{\infty}\tbinom{n+\lambda -1}{\lambda }\tfrac{n^{k}\left[
(n-\gamma )\left\vert a_{n}\right\vert +(n+\gamma )\left\vert
b_{n}\right\vert \right] }{(1-\gamma )(2n-1)(n-1)!}\leq 2,  \label{2`1}
\end{equation}%
then $f$ is sense-preserving, harmonic univalent in $\mathbb{U}$ and $f\in 
\mathcal{E}S_{\mathcal{H}}(k,\lambda ,\gamma )$.
\end{theorem}

\begin{proof}
Let suppose that $z_{1}\neq z_{2}.$ Thus, we get%
\begin{eqnarray*}
\left\vert \frac{f(z_{1})-f(z_{2})}{h(z_{1})-h(z_{2})}\right\vert &\geq
&1-\left\vert \frac{g(z_{1})-g(z_{2})}{h(z_{1})-h(z_{2})}\right\vert
=1-\left\vert \frac{\Sigma_{n=1}^{\infty}b_{n}\left(
z_{1}^{n}-z_{2}^{n}\right) }{\left( z_{1}-z_{2}\right) +\Sigma_{n=2}^{\infty}a_{n}\left( z_{1}^{n}-z_{2}^{n}\right) }\right\vert
\\
&& \\
&>&1-\frac{\Sigma_{n=1}^{\infty}n\left\vert
b_{n}\right\vert }{1-\Sigma_{n=2}^{\infty}n\left\vert
a_{n}\right\vert }
\end{eqnarray*}
\begin{eqnarray*}
&\geq &1-\frac{\Sigma_{n=1}^{\infty}\binom{n+\lambda -1}{%
\lambda }\frac{(n+\gamma )\left\vert b_{n}\right\vert n^{k}}{(1-\gamma
)(2n-1)(n-1)!}\left\vert b_{n}\right\vert }{1-\Sigma_{n=2}^{\infty}\binom{n+\lambda -1}{\lambda }\frac{(n-\gamma )n^{k}}{(1-\gamma
)(2n-1)(n-1)!}\left\vert a_{n}\right\vert } \\
&& \\
&\geq &0,
\end{eqnarray*}
which proves univalence. Since 
\begin{eqnarray*}
\left\vert h^{\prime }(z)\right\vert &\geq &1-\Sigma_{n=2}^{\infty}n\left\vert a_{n}\right\vert \left\vert z\right\vert ^{n-1}>1-\Sigma_{n=2}^{\infty}\tbinom{n+\lambda -1}{\lambda }\tfrac{(n-\gamma )n^{k}}{(1-\gamma )(2n-1)(n-1)!}\left\vert a_{n}\right\vert \\
&& \\
&\geq &\Sigma_{n=1}^{\infty}\tbinom{n+\lambda -1}{\lambda }%
\tfrac{(n+\gamma )\left\vert b_{n}\right\vert n^{k}}{(1-\gamma )(2n-1)(n-1)!}%
\left\vert b_{n}\right\vert >\Sigma_{n=1}^{\infty}
n\left\vert b_{n}\right\vert \left\vert z\right\vert ^{n-1} \\
&& \\
&\geq &\left\vert g^{\prime }(z)\right\vert ,
\end{eqnarray*}
$f$ is sense preserving in $\mathbb{U}$. To prove that $f\in \mathcal{E}S_{\mathcal{H}}(k,\lambda ,\gamma )$, we must show that if $(\ref{2`1})$ holds, then the required condition $(\ref{5})$ is fulfilled. By the fact that $\Re (w)\geq \gamma $ iff $\left\vert 1-\gamma +w\right\vert \geq \left\vert 1+\gamma -w\right\vert $,
we must find that
\begin{equation*}
\begin{array}{l}
\left\vert (1-\gamma )C_{\lambda }^{k}\mathcal{ERF}(z)+z\left( C_{\lambda
}^{k}\mathcal{ERF}(z)\right) _{z}-\overline{z}\left( C_{\lambda }^{k}%
\mathcal{ERF}(z)\right) _{\overline{z}}\right\vert \\ 
\\ 
-\left\vert (1+\gamma )C_{\lambda }^{k}\mathcal{ERF}(z)-z\left( C_{\lambda
}^{k}\mathcal{ERF}(z)\right) _{z}+\overline{z}\left( C_{\lambda }^{k}%
\mathcal{ERF}(z)\right) _{\overline{z}}\right\vert \geq 0.%
\end{array}%
\end{equation*}
Hence, we have%
\begin{equation*}
\begin{array}{l}
\left\vert (1-\gamma )C_{\lambda }^{k}\mathcal{ERF}(z)+z\left( C_{\lambda
}^{k}\mathcal{ERF}(z)\right) _{z}-\overline{z}\left( C_{\lambda }^{k}%
\mathcal{ERF}(z)\right) _{\overline{z}}\right\vert \\ 
\\ 
-\left\vert (1+\gamma )C_{\lambda }^{k}\mathcal{ERF}(z)-z\left( C_{\lambda
}^{k}\mathcal{ERF}(z)\right) _{z}+\overline{z}\left( C_{\lambda }^{k}%
\mathcal{ERF}(z)\right) _{\overline{z}}\right\vert \\ 
\\ 
=\left\vert (1-\gamma )z+(1-\gamma )\Sigma_{n=2}^{\infty}
\binom{n+\lambda -1}{\lambda }\frac{(-1)^{n-1}n^{k}}{(2n-1)(n-1)!}%
a_{n}z^{n}+(1-\gamma )\Sigma_{n=1}^{\infty}\binom{%
n+\lambda -1}{\lambda }\frac{(-1)^{n-1}n^{k}}{(2n-1)(n-1)!}\overline{%
b_{n}z^{n}}\right. \\ 
\\ 
+\left. z+\Sigma_{n=2}^{\infty}\binom{n+\lambda -1}{%
\lambda }\frac{(-1)^{n-1}n^{k+1}}{(2n-1)(n-1)!}a_{n}z^{n}-\Sigma_{n=1}^{\infty}\binom{n+\lambda -1}{\lambda }\frac{(-1)^{n-1}n^{k+1}}{%
(2n-1)(n-1)!}\overline{b_{n}z^{n}}\right\vert \\ 
\\ 
-\left\vert (1+\gamma )z+(1+\gamma )\Sigma_{n=2}^{\infty}
\binom{n+\lambda -1}{\lambda }\frac{(-1)^{n-1}n^{k}}{(2n-1)(n-1)!}
a_{n}z^{n}+(1+\gamma )\Sigma_{n=1}^{\infty}\binom{n+\lambda -1}{\lambda }\frac{(-1)^{n-1}n^{k}}{(2n-1)(n-1)!}\overline{%
b_{n}z^{n}}\right. \\ 
\\ 
\left. -z-\Sigma_{n=2}^{\infty}\binom{n+\lambda -1}{%
\lambda }\frac{(-1)^{n-1}n^{k+1}}{(2n-1)(n-1)!}a_{n}z^{n}+\Sigma_{n=1}^{\infty}\binom{n+\lambda -1}{\lambda }\frac{(-1)^{n-1}n^{k+1}}{%
(2n-1)(n-1)!}\overline{b_{n}z^{n}}\right\vert%
\end{array}
\end{equation*}
\begin{equation*}
\begin{array}{l}
=\left\vert (2-\gamma )z+\Sigma_{n=2}^{\infty}\binom{
n+\lambda -1}{\lambda }\frac{(-1)^{n-1}(n+1-\gamma )n^{k}}{(2n-1)(n-1)!}%
a_{n}z^{n}-\Sigma_{n=1}^{\infty}\binom{n+\lambda -1}{%
\lambda }\frac{(-1)^{n-1}(n-1+\gamma )n^{k}}{(2n-1)(n-1)!}\overline{%
b_{n}z^{n}}\right\vert \\ 
\\ 
-\left\vert \gamma z-\Sigma_{n=2}^{\infty}\binom{%
n+\lambda -1}{\lambda }\frac{(-1)^{n-1}(n-1-\gamma )n^{k}}{(2n-1)(n-1)!}%
a_{n}z^{n}+\Sigma_{n=1}^{\infty}\binom{n+\lambda -1}{%
\lambda }\frac{(-1)^{n-1}(n+1+\gamma )n^{k}}{(2n-1)(n-1)!}\overline{%
b_{n}z^{n}}\right\vert \\ 
\\ 
\geq 2\left\vert z\right\vert \left\{ 1-\gamma -\Sigma_{n=2}^{\infty}
\binom{n+\lambda -1}{\lambda }\frac{n^{k}(n-\gamma )\left\vert
a_{n}\right\vert }{(2n-1)(n-1)!}-(1+\gamma )b_{1}+\Sigma_{n=2}^{\infty}\binom{n+\lambda -1}{\lambda }\frac{n^{k}(n+\gamma )\left\vert
b_{n}\right\vert }{(2n-1)(n-1)!}\right\} .%
\end{array}%
\end{equation*}%
The above expression is non-negative by $(\ref{2`1})$ and so $f\in \mathcal{E%
}S_{\mathcal{H}}(k,\lambda ,\gamma ).$
\end{proof}

In the following theorem, we attempt to find the necessary and sufficient conditions
for $f$ of the form $(\ref{x})$.

\begin{theorem}
\label{t2} $f\in \mathcal{E}S_{\overline{\mathcal{H}}}(k,\lambda ,\gamma )$ if and only
if the condition $(\ref{2`1})$ holds true.
\end{theorem}

\begin{proof}
In view of Theorem \ref{t1}, we must show that each function $f\in \mathcal{E}%
S_{\overline{\mathcal{H}}}(k,\lambda ,\gamma )$ fulfills the inequality (\ref{2`1}). We deduce that the condition $(\ref{2`1})$ is
equivalent to 
\begin{equation}
\Re \left[ \frac{(1-\gamma )z-\Sigma_{n=2}^{\infty}\binom{n+\lambda -1%
}{\lambda }\frac{\left( n-\gamma \right) n^{k}}{(2n-1)(n-1)!}\left\vert
a_{n}\right\vert z^{n}-\Sigma_{n=1}^{\infty}\binom{n+\lambda -1}{%
\lambda }\frac{\left( n+\gamma \right) n^{k}}{(2n-1)(n-1)!}\left\vert
b_{n}\right\vert \overline{z}^{n}}{z-\Sigma_{n=2}^{\infty}\frac{n^{k}}{%
(2n-1)(n-1)!}\left\vert a_{n}\right\vert z^{n}+\Sigma_{n=1}^{\infty}
\frac{n^{k}}{(2n-1)(n-1)!}\left\vert b_{n}\right\vert \overline{z}^{n}}%
\right] \geq 0\ \ \,\,(z\in \mathbb{U}).  \label{A}
\end{equation}%
Thus, letting $z=r~(0\leq r<1)$ by (\ref{A}), we must arrive
\begin{equation}
\frac{1-\gamma -\Sigma_{n=2}^{\infty}\binom{n+\lambda -1}{\lambda }%
\frac{\left( n-\gamma \right) n^{k}}{(2n-1)(n-1)!}\left\vert
a_{n}\right\vert r^{n-1}-\Sigma_{n=2}^{\infty}\binom{n+\lambda -1}{%
\lambda }\frac{\left( n+\gamma \right) n^{k}}{(2n-1)(n-1)!}\left\vert
b_{n}\right\vert r^{n-1}}{1-\Sigma_{n=2}^{\infty}\frac{n^{k}}{%
(2n-1)(n-1)!}\left\vert a_{n}\right\vert r^{n-1}+\Sigma_{n=1}^{\infty}
\frac{n^{k}}{(2n-1)(n-1)!}\left\vert b_{n}\right\vert r^{n-1}}\geq 0.
\label{2`6}
\end{equation}%
Let the condition (\ref{2`1}) does not hold. The denominator of the LHS cannot vanish for $%
r\in \left( 0,1\right) .$ Moreover, it is positive for $r=0,$ and in
consequence for $r\in \left[ 0,1\right) .$ This contradicts the fact $f\in \mathcal{E}S_{\overline{\mathcal{H}}}(k,\lambda ,\gamma )$. Thus, the proof is complete.
\end{proof}

Next, we attempt to give the distortion bounds for functions in $\mathcal{E}S_{\overline{\mathcal{H}}}(k,\lambda ,\gamma )$.

\begin{theorem}
If $f\in \mathcal{E}S_{\overline{\mathcal{H}}}(k,\lambda ,\gamma )$, then%
\begin{equation*}
\begin{array}{c}
\left\vert f(z)\right\vert \leq \left( 1+\left\vert b_{1}\right\vert \right)
r+\left( \frac{3(1-\gamma )}{2^{k}(2-\gamma )(1+\lambda )}-\frac{3(1+\gamma )%
}{2^{k}(2-\gamma )(1+\lambda )}\left\vert b_{1}\right\vert \right) r^{2} \\ 
\\ 
\left( \left\vert z\right\vert =r~,0\leq r<1\right) 
\end{array}%
\end{equation*}%
and%
\begin{equation*}
\begin{array}{c}
\left\vert f(z)\right\vert \geq \left( 1-\left\vert b_{1}\right\vert \right)
r-\left( \frac{3(1-\gamma )}{2^{k}(2-\gamma )(1+\lambda )}-\frac{3(1+\gamma )%
}{2^{k}(2-\gamma )(1+\lambda )}\left\vert b_{1}\right\vert \right) r^{2} \\ 
\\ 
\left( \left\vert z\right\vert =r~,0\leq r<1\right) .%
\end{array}%
\end{equation*}
\end{theorem}

\begin{proof}
Since $f\in \mathcal{E}S_{\overline{\mathcal{H}}}(k,\lambda ,\gamma )$, then
by taking the absolute value of $f$, we obtain%
\begin{eqnarray*}
\left\vert f(z)\right\vert &\geq &\left( 1-\left\vert b_{1}\right\vert
\right) r-\Sigma_{n=2}^{\infty}\left( \left\vert a_{n}\right\vert
+\left\vert b_{n}\right\vert \right) r^{n} \\
&& \\
&\geq &\left( 1-\left\vert b_{1}\right\vert \right)
r-\Sigma_{n=2}^{\infty}\left( \left\vert a_{n}\right\vert +\left\vert
b_{n}\right\vert \right) r^{2} \\
&& \\
&\geq &\left( 1-\left\vert b_{1}\right\vert \right) r-\tfrac{3(1-\gamma )}{%
2^{k}(2-\gamma )(1+\lambda )}\Sigma_{n=2}^{\infty}\left( \tfrac{%
2^{k}(2-\gamma )(1+\lambda )}{3(1-\gamma )}\left\vert a_{n}\right\vert +%
\tfrac{2^{k}(2-\gamma )(1+\lambda )}{3(1-\gamma )}\left\vert
b_{n}\right\vert \right) r^{2} \\
&& \\
&\geq &\left( 1-\left\vert b_{1}\right\vert \right) r-\tfrac{3(1-\gamma )}{%
2^{k}(2-\gamma )(1+\lambda )}\Sigma_{n=2}^{\infty}\tbinom{n+\lambda -1%
}{\lambda }\left( \tfrac{n^{k}(n-\gamma )}{(1-\gamma )(2n-1)(n-1)!}%
\left\vert a_{n}\right\vert +\tfrac{n^{k}(n+\gamma )}{(1-\gamma )(2n-1)(n-1)!%
}\left\vert b_{n}\right\vert \right) r^{2} \\
&& \\
&\geq &\left( 1-\left\vert b_{1}\right\vert \right) r-\frac{3(1-\gamma )}{%
2^{k}(2-\gamma )(1+\lambda )}\left[ 1-\frac{1+\gamma }{1-\gamma }\left\vert
b_{1}\right\vert \right] r^{2} \\
&& \\
&\geq &\left( 1-\left\vert b_{1}\right\vert \right) r-\left( \frac{%
3(1-\gamma )}{2^{k}(2-\gamma )(1+\lambda )}-\frac{3(1+\gamma )}{%
2^{k}(2-\gamma )(1+\lambda )}\left\vert b_{1}\right\vert \right) r^{2}.
\end{eqnarray*}

The proof for the right hand inequality is similar, hence we omit it.
\end{proof}

\begin{corollary}
Let $f$ be of the form $(\ref{x})$ so that $f\in \mathcal{E}S_{\overline{%
\mathcal{H}}}(k,\lambda ,\gamma )$. Then%
\begin{equation*}
\left\{ w:\left\vert w\right\vert <\tfrac{2^{k+1}(1+\lambda )-3-\left[
2^{k}(1+\lambda )-3\right] \gamma }{2^{k}(2-\gamma )(1+\lambda )}-\tfrac{%
2^{k+1}(1+\lambda )-3-\left[ 2^{k}(1+\lambda )+3\right] \gamma }{%
2^{k}(2-\gamma )(1+\lambda )}\left\vert b_{1}\right\vert \right\} \subset f(%
\mathbb{U}).
\end{equation*}
\end{corollary}

\begin{theorem}
Let $f$ be given by (\ref{x}). Then $f\in \mathcal{E}S_{\overline{\mathcal{H}%
}}(k,\lambda ,\gamma )$ if and only if
\begin{equation*}
f(z)=\Sigma_{n=1}^{\infty}\left(
X_{n}h_{n}(z)+Y_{n}g_{n}(z)\right) ,
\end{equation*}
where
\begin{equation*}
\begin{array}{c}
h_{1}(z)=z,\ \ \ \ \ h_{n}(z)=z-\frac{(1-\gamma )(2n-1)(n-1)!}{%
n^{k}(n-\gamma )\binom{n+\lambda -1}{\lambda }}z^{n}\ \ (n\geq 2), \\ 
\\ 
g_{n}(z)=z+\frac{(1-\gamma )(2n-1)(n-1)!}{n^{k}(n+\gamma )\binom{n+\lambda -1%
}{\lambda }}\overline{z}^{n}\ \ \ (n\geq 1), \\ 
\\ 
\Sigma_{n=1}^{\infty}\left( X_{n}+Y_{n}\right)
=1,X_{n}\geq 0,Y_{n}\geq 0.%
\end{array}%
\end{equation*}
In particular, the extreme points of $\mathcal{E}S_{\overline{%
\mathcal{H}}}(k,\lambda ,\gamma )$ \textit{are }$\left\{ h_{n}\right\} $%
\textit{\ and }$\left\{ g_{n}\right\} $\textit{.}
\end{theorem}
\begin{proof}
For functions $f$ of the form (\ref{x}), we may write 
\begin{eqnarray*}
f(z) &=&\Sigma_{n=1}^{\infty}\left(
X_{n}h_{n}(z)+Y_{n}g_{n}(z)\right) \\
&& \\
&=&\Sigma_{n=1}^{\infty}\left( X_{n}+Y_{n}\right) z-%
\Sigma_{n=2}^{\infty}\tfrac{(1-\gamma )(2n-1)(n-1)!}{%
n^{k}(n-\gamma )\binom{n+\lambda -1}{\lambda }}X_{n}z^{n} \\
&& \\
&&+\Sigma_{n=1}^{\infty}\tfrac{(1-\gamma )(2n-1)(n-1)!}{%
n^{k}(n+\gamma )\binom{n+\lambda -1}{\lambda }}Y_{n}\overline{z}^{n}.
\end{eqnarray*}%
Then%
\begin{eqnarray*}
&&\Sigma_{n=2}^{\infty}\tbinom{n+\lambda -1}{\lambda }%
\tfrac{n^{k}(n-\gamma )}{(1-\gamma )(2n-1)(n-1)!}\left( \tfrac{(1-\gamma
)(2n-1)(n-1)!}{n^{k}(n-\gamma )\binom{n+\lambda -1}{\lambda }}X_{n}\right) \\
&& \\
&&+\Sigma_{n=1}^{\infty}\tbinom{n+\lambda -1}{\lambda }%
\tfrac{n^{k}(n+\gamma )}{(1-\gamma )(2n-1)(n-1)!}\left( \tfrac{(1-\gamma
)(2n-1)(n-1)!}{n^{k}(n+\gamma )\binom{n+\lambda -1}{\lambda }}Y_{n}\right) \\
&& \\
&=&\Sigma_{n=2}^{\infty}X_{n}+\Sigma_{n=1}^{\infty}Y_{n}=1-X_{1}\leq 1,\text{ and so }f\in \mathcal{E}S_{\overline{%
\mathcal{H}}}(k,\lambda ,\gamma ).
\end{eqnarray*}
Conversely, if $f\in \mathcal{E}S_{\overline{\mathcal{H}}}(k,\lambda ,\gamma
)$, then
\begin{equation*}
\left\vert a_{n}\right\vert \leq \binom{n+\lambda -1}{\lambda }\dfrac{%
n^{k}(n-\gamma )}{(1-\gamma )(2n-1)(n-1)!}
\end{equation*}%
and 
\begin{equation*}
\left\vert b_{n}\right\vert \leq \binom{n+\lambda -1}{\lambda }\dfrac{%
n^{k}(n+\gamma )}{(1-\gamma )(2n-1)(n-1)!}.
\end{equation*}%
In particular, setting 
\begin{equation*}
X_{n}=\frac{(1-\gamma )(2n-1)(n-1)!}{n^{k}(n-\gamma )\binom{n+\lambda -1}{%
\lambda }}\left\vert a_{n}\right\vert \text{ }(n\geq 2),
\end{equation*}%
\begin{equation*}
Y_{n}=\frac{(1-\gamma )(2n-1)(n-1)!}{n^{k}(n+\gamma )\binom{n+\lambda -1}{%
\lambda }}\left\vert b_{n}\right\vert \ \ (n\geq 1),
\end{equation*}%
and 
\begin{equation*}
X_{1}=1-\left( \Sigma_{n=2}^{\infty}X_{n}+\Sigma_{n=1}^{\infty}Y_{n}\right) ,
\end{equation*}%
where $X_{1}\geq 0$, we then see that%
\begin{equation*}
f(z)=\Sigma_{n=1}^{\infty }\left( X_{n}h_{n}(z)+Y_{n}g_{n}(z)\right) .
\end{equation*}
\end{proof}

\begin{theorem}
The class $\mathcal{E}S_{\overline{\mathcal{H}}}(k,\lambda ,\gamma)$ is closed under convex combinations.
\end{theorem}

\begin{proof}
Assume $f_{i}\in $ $\mathcal{E}S_{\overline{\mathcal{H}}}(k,\lambda ,\gamma )$, where
\begin{equation*}
f_{i}(z)=z+\Sigma_{n=2}^{\infty}(-1)^{n}\left\vert
a_{n_{i}}\right\vert z^{n}+\Sigma_{n=1}^{\infty}
(-1)^{n-1}\left\vert b_{n_{i}}\right\vert \overline{z}^{n} \ \ \ (i=1,2,...).
\end{equation*}%
Thus, by $(\ref{2`1})$,%
\begin{equation}
\Sigma_{n=2}^{\infty}\tbinom{n+\lambda -1}{\lambda }\tfrac{(n-\gamma )n^{k}}{(1-\gamma )(2n-1)(n-1)!}\left\vert a_{n_{i}}\right\vert +
\Sigma_{n=1}^{\infty}\tbinom{n+\lambda -1}{\lambda }\tfrac{
(n+\gamma )n^{k}}{(1-\gamma )(2n-1)(n-1)!}\left\vert b_{n_{i}}\right\vert
\leq 1.  \label{eqconcomb}
\end{equation}%
For $\Sigma_{i=1}^{\infty}t_{i}=1$ $\left( 0\leq
t_{i}\leq 1\right) $, the convex combination of $f_{i}$ can be expressed by%
\begin{equation*}
\Sigma_{i=1}^{\infty}t_{i}f_{i}(z)=z+\Sigma_{n=2}^{\infty}(-1)^{n}\left( \Sigma_{i=1}^{\infty}
t_{i}\left\vert a_{n_{i}}\right\vert \right) z^{n}+\Sigma_{n=1}^{\infty}(-1)^{n-1}\left( \Sigma_{i=1}^{\infty}t_{i}\left\vert b_{n_{i}}\right\vert \right) \overline{z}^{n}.
\end{equation*}
Finally, we deduce from (\ref{eqconcomb}),
\begin{eqnarray*}
&&\Sigma_{n=2}^{\infty}\tbinom{n+\lambda -1}{\lambda }%
\tfrac{(n-\gamma )n^{k}}{(1-\gamma )(2n-1)(n-1)!}\left( \Sigma_{i=1}^{\infty}t_{i}\left\vert a_{n_{i}}\right\vert \right) \\
&& \\
&&+\Sigma_{n=1}^{\infty}\tbinom{n+\lambda -1}{\lambda }%
\frac{(n+\gamma )n^{k}}{(1-\gamma )(2n-1)(n-1)!}\left( \Sigma_{i=1}^{\infty}t_{i}\left\vert b_{n_{i}}\right\vert \right) \\
&& \\
&\leq &\Sigma_{i=1}^{\infty}t_{i}\left( \Sigma_{n=2}^{\infty}\tbinom{n+\lambda -1}{\lambda }\tfrac{(n-\gamma )n^{k}%
}{(1-\gamma )(2n-1)(n-1)!}\left\vert a_{n_{i}}\right\vert \right. \\
&& \\
&&\left. +\Sigma_{n=1}^{\infty}\tbinom{n+\lambda -1}{%
\lambda }\tfrac{(n+\gamma )n^{k}}{(1-\gamma )(2n-1)(n-1)!}\left\vert
b_{n_{i}}\right\vert \right) \\
&& \\
&\leq &\Sigma_{i=1}^{\infty}t_{i}=1.
\end{eqnarray*}

This is the condition required by $(\ref{2`1})$ and so $\Sigma_{i=1}^{\infty}t_{i}f_{i}(z)\in $ $\mathcal{E}S_{\overline{%
\mathcal{H}}}(k,\lambda ,\gamma ).$
\end{proof}

Next we present the closure properties of $\mathcal{E}S_{%
\overline{\mathcal{H}}}(k,\lambda ,\gamma )$ under the generalized
Bernardi-Libera-Livingston integral operator $\mathfrak{F}_{c}(f)$
defined by%
\begin{equation}
\mathfrak{F}_{c}(f(z))=\frac{1+c}{z^{c}}\int\limits_{0}^{z}t^{c-1}f(t)dt\ \
\ \ \ \ (c\geq 0).  \label{O}
\end{equation}%
While the operator $\mathfrak{F}_{c}$, when $c\in \mathbb{N}$, was
defined by Bernardi \cite{ber}, $\mathfrak{F}_{1}$ was investigated earlier by Libera \cite{lib} and Livingston \cite{liv}.

\begin{theorem}
Let $f\in \mathcal{E}S_{\overline{\mathcal{H}}}(k,\lambda ,\gamma )$. Then $%
\mathfrak{F}_{c}(f)$ belongs to the family $\mathcal{E}S_{\overline{\mathcal{H%
}}}(k,\lambda ,\gamma )$.
\end{theorem}

\begin{proof}
By making use of the representation of $\mathfrak{F}_{c}(f)$, we arrive
\begin{equation*}
\begin{array}{ll}
\mathfrak{F}_{c}(f(z)) & =\dfrac{1+c}{z^{c}}\int\limits_{0}^{z}t^{c-1}\left(
h(t)+\overline{g(t)}\right) dt\  \\ 
&  \\ 
& =\dfrac{1+c}{z^{c}}\int\limits_{0}^{z}t^{c-1}\left(
t+\Sigma_{n=2}^{\infty}(-1)^{n}\left\vert a_{n}\right\vert
t^{n}\right) +\overline{\int\limits_{0}^{z}t^{c-1}\left(
\Sigma_{n=1}^{\infty}(-1)^{n-1}\left\vert b_{n}\right\vert
t^{n}\right) }dt\  \\ 
&  \\ 
& =z+\Sigma_{n=2}^{\infty}\frac{c+1}{c+n}(-1)^{n}\left\vert
a_{n}\right\vert z^{n}+\Sigma_{n=1}^{\infty}\overline{\frac{c+1}{c+n}%
(-1)^{n-1}\left\vert b_{n}\right\vert z^{n}}.%
\end{array}
\end{equation*}
Hence,
\begin{equation*}
\begin{array}{l}
\Sigma_{n=2}^{\infty}\binom{n+\lambda -1}{\lambda }\left( \frac{%
n-\gamma }{1-\gamma }\frac{c+1}{c+n}\left\vert a_{n}\right\vert +\frac{%
n+\gamma }{1-\gamma }\frac{c+1}{c+n}\left\vert b_{n}\right\vert \right) 
\frac{1}{(2n-1)(n-1)!} \\ 
\\ 
\leq \Sigma_{n=2}^{\infty}\binom{n+\lambda -1}{\lambda }\left( \frac{
n-\gamma }{1-\gamma }\left\vert a_{n}\right\vert +\frac{n+\gamma }{1-\gamma }
\left\vert b_{n}\right\vert \right) \frac{1}{(2n-1)(n-1)!} \\ 
\\ 
\leq 1-\frac{1+\gamma }{1-\gamma }\left\vert b_{1}\right\vert .%
\end{array}%
\end{equation*}

Since $f\in \mathcal{E}S_{\overline{\mathcal{H}}}(k,\lambda ,\gamma )$,
therefore by Theorem 2, $\mathfrak{F}_{c}(f)\in \mathcal{E}S_{\overline{%
\mathcal{H}}}(k,\lambda ,\gamma )$.
\end{proof}

\end{document}